\documentclass[letterpaper, 10 pt, conference]{ieeeconf}

\IEEEoverridecommandlockouts

\usepackage{balance}

\usepackage[english]{babel}
\usepackage{amssymb,amsmath} 
\usepackage[latin1,utf8]{inputenc}
\usepackage{graphicx}
\usepackage{epstopdf}
\usepackage{subcaption}
\usepackage[section]{placeins}
\usepackage[]{algorithm}
\usepackage{algorithmic}
\usepackage{cite}
\usepackage[utf8]{inputenc}
\usepackage[mathscr]{euscript}
\usepackage{hyperref}
\usepackage{color}
\usepackage{bm}

\usepackage[normalem]{ulem}

\usepackage{tikz}
\usetikzlibrary{fit,positioning,arrows,automata}

\def\minwrt[#1]{\underset{#1}{\text{minimize }}}
\def\argminwrt[#1]{\underset{#1}{\text{arg min }}}
\def\maxwrt[#1]{\underset{#1}{\text{maximize }}}
\def\argmaxwrt[#1]{\underset{#1}{\text{arg max }}}
\def\maxemphwrt[#1]{\underset{#1}{\text{\emph{maximize} }}}
\newcommand{\ett}{{\bf 1}}

\newcommand{\mR}{{\mathbb R}}
\newcommand{\diag}{{\rm diag}}

\newtheorem{remark}{Remark}
\newtheorem{proposition}{Proposition}

\newtheorem{definition}{Definition}

\newcommand{\trace}[1]{\text{tr}\left(#1\right)}
\newcommand{\supp}{\rm supp}

\def\ccP{{\mathcal{P}}}
\def\ccW{{\mathcal{W}}}
\def\ccN{{\mathcal{N}}}

\def\RN{{\mathbb{N}}}
\def\RR{{\mathbb{R}}}

\title{\LARGE \bf
Estimating ensemble flows on a hidden Markov chain
}

\author{Isabel Haasler, Axel Ringh, Yongxin Chen, and Johan Karlsson
\thanks{This work was supported by the Swedish Research Council (VR), grant 2014-5870, SJTU-KTH cooperation grant and the NSF under grant 1901599.
}
\thanks{I.~Haasler, A.~Ringh, and J.~Karlsson are with the Division of Optimization and Systems Theory, Department of Mathematics, KTH Royal Institute of Technology, Stockholm, Sweden. {\tt\small \{haasler, aringh\}@kth.se}, {\tt\small johan.karlsson@math.kth.se}}%
\thanks{Y. Chen is with the School of Aerospace Engineering,
Georgia Institute of Technology, Atlanta, GA, USA. {\tt\small yongchen@gatech.edu}}%
}

\begin{document}

\maketitle
\thispagestyle{empty}
\pagestyle{empty}

\begin{abstract}

We propose a new framework to estimate the evolution
of an ensemble of indistinguishable agents on a hidden Markov chain using only aggregate output data.
This work can be viewed as an extension of the recent developments in optimal mass transport and Schrödinger bridges to the finite state space hidden Markov chain setting.
The flow of the ensemble is estimated by solving a maximum likelihood problem, which has a convex formulation at the infinite-particle limit, and we develop a fast numerical algorithm for it.
We illustrate in two numerical examples how this framework can be used to track the flow of identical and indistinguishable dynamical systems.

\end{abstract}

\section{Introduction}

State tracking of a set of agents is an important issue in many areas, e.g., target tracking, see \cite{blackman1999design} and references therein.
In this case, one is often interested in tracking one single or a set of multiple distinct targets.
However, in many applications information for each agent may not be available, e.g., if the population is too large to track every single agent, as in many biological systems, or due to data privacy \cite{King97}.
In this work, we thus consider tracking the evolution of a finite ensemble of indistinguishable agents. Based on reduced and incomplete measurements of the whole population at different time points, we aim to recover an estimate of the discrete-time flow of the ensemble. Related state estimation problems for a continuum of agents and in continuous time have been considered in \cite{chen2018state,zeng2016ensemble} (see also \cite{chen2018measure}).

In this work, we use a hidden Markov model (HMM) to describe the particle flows and aggregate observations, similar to \cite{bernstein2016}, and seek the most likely paths that the agents have taken. These paths are found by maximizing the log-likelihood function of the flow, subject to the constraint that the flow matches the given measurements. This gives rise to a convex maximum entropy type optimization problem, and we derive an efficient algorithm for solving it.

The problem of finding the most likely path for the evolution of a distribution is related to a discrete Schrödinger bridge problem \cite{pavon2010discrete}.

Schrödinger's thought experiment \cite{schrodinger1931} has indirectly given rise to the concept of reciprocal processes 
\cite{jamison1974reciprocal, levy1990modeling, carravetta2012modelling}, which connects this work to 
tracking of moving objects using reciprocal processes \cite{fanaswala2011destination, white2013maximum, stamatescu2018track}.
However, as mentioned before, we consider estimating the flow of an ensemble rather than single target tracking.

The outline of the paper is as follows: Section II presents background material, in particular on HMMs, Schrödinger bridges, and optimal mass transport. In section III we derive the maximum likelihood problem for a Markov chain with a known initial and final distribution,
and relate it to prior work on the Schrödinger bridge problem \cite{pavon2010discrete}. In Section IV, which contains the main contribution, we extend this maximum entropy framework to HMMs with indirect and noisy observations. Moreover, we derive the corresponding maximum likelihood problem, and develop a fast iterative algorithm for solving it.
The method is demonstrated on two examples
in section V,
and section VI contains conclusions and future directions.
Some proofs are deferred to the appendix for improved readability.  

\section{Background}
\subsection{Notation}
By $./$, $\odot$, $\log(\cdot)$, and $\exp(\cdot)$ we denote elementwise division, multiplication, logarithm, and exponential of matrices and vectors. Moreover, by $\supp(\cdot)$ we denote the support of a matrix, i.e., the non-zero elements.

\subsection{Hidden Markov chains}\label{ssec:HMM}
In this work, we consider hidden Markov models for stochastic modeling of a group of indistinguishable agents/particles. For an introduction to HMMs, see, e.g., \cite{rabiner1989tutorial, ghahramani2001introduction}. An HMM is a structure that consists of two stochastic processes. The first part is a Markov chain that evolves over a hidden set of states $X= \left\{ X_1, X_2, \dots, X_n \right\}$ and is used to model the unobserved, underlying state of the system. We denote the state at time $t$ by $q_t$. The stochastic state transitions are encoded in the state transition matrix $A=\left[a_{ij}\right]_{i,j = 1}^n$, where $a_{ij} = P( q_{t+1}=X_j | q_t = X_i )$. The second part is an observation process providing partial and noisy information of the underlying process; here we use the observation symbols $Y=\left\{ Y_1, Y_2,\dots, Y_m \right\}$. Moreover, the observation process is also Markovian with respect to the underlying state in the hidden Markov chain, i.e., the observation probabilities
can be summarized in a matrix $B\in\RR^{n\times m}$ with elements $b_{jk}= P( Y_k \text{ at } t | q_t= X_j )$.
\subsection{Schrödinger bridges and large deviations}
\label{ssec:bridge}
In the early 1930s, Schrödinger discussed the problem of determining the evolution of particles between two observed distributions \cite{schrodinger1931}.
Assuming a cloud of independent Brownian particles is observed at time instance $t=0$, the expected distribution at $t=1$ would be described by
\begin{equation} \label{eq:schrodinger_final_distribution}
\rho_1(x_1) = \int_{\RR^n} q_\epsilon(0,x_0,1,x_1) \rho_0(x_0) dx_0,
\end{equation}
where $q_\varepsilon$ is the Brownian transition probability kernel
\begin{equation*}
q_\epsilon(s,x,t,y) =  \frac{1}{\left(2\pi(t-s)\epsilon\right)^{n/2}}  \exp\left( -\frac{\|x-y\|^2}{2(t-s)\epsilon} \right),
\end{equation*}
and where the parameter $\epsilon$ denotes a diffusion coefficient.
Schrödinger studied the problem where the observed particle distribution differs from the expected distribution \eqref{eq:schrodinger_final_distribution}. The most likely particle evolution connecting, hence bridging, the two marginals is called the Schrödinger bridge.

The Schrödinger bridge problem was later formulated in the context of large deviation theory \cite[Sec.~II.1.3]{Follmer88}, the study of rare events in the sense of deviations from the law of large numbers \cite{ellis2006book, dembo2009large}. As the number of trials (or particles) goes to infinity, the probability of such rare events approaches zero. Large deviation theory studies the rate of this decay, which can often be characterized by the exponential of a so called rate function.

Modeling the particle evolutions as independent identically distributed random variables on path space, a Schrödinger bridge is a probability measure $\ccP$ on path space that is most likely to describe the rare event of observing the two particle distributions. 
Such a measure is obtained by minimizing the corresponding rate function, which turns out to be the relative entropy with respect to the underlying probability law of the Brownian motion. 
In other words, $\ccP$  is the measure that is ``most similar'' to the Wiener measure $\ccW$ in the sense that it minimizes the relative entropy \cite{leonard2013schrodinger}
\begin{equation} \label{eq:schrodingerbridge}
H(\ccP \mid \ccW) = \int \log\left(\frac{d\ccP}{d\ccW} \right) d\ccP
\end{equation}
over all probability measures that are absolutely continuous with respect to $\ccW$ and have the given particle distributions as marginals.
The Schrödinger bridge can be constructed from the solution to a certain system of equations, called the Schrödinger system.
A space and time discrete Schrödinger bridge problem for Markov chains is analysed in \cite{pavon2010discrete, Georgiou2015discreteSB,chen2017robust}.
\subsection{Optimal mass transport}\label{subsec:OMT_intro}
Another recently established connection of Schrödinger bridges is to the problem of optimal mass transport (OMT) \cite{chen2016relation, chen2016hilbert, leonard2013schrodinger, Mikami2004}.
As the diffusion coefficient $\epsilon$ in \eqref{eq:schrodinger_final_distribution} approaches $0$, the solution to the Schrödinger bridge problem tends to the solution to a corresponding optimal mass transport problem\cite{leonard2013schrodinger}.
Moreover, the Schrödinger bridge formulation is a regularization of OMT, as it is strictly convex and therefore guarantees a unique solution.

We introduce a discretized formulation of the OMT problem. For an extensive discussion of OMT see, e.g., \cite{villani2008optimal}. Consider a discretization $\left\{x_1,\dots,x_n\right\}$ of a compact space $X$ and two distributions $\mu_0,\mu_1\in\RR^n$ defined on this discretization. Given a cost matrix $C=[c_{ij}]_{i,j=1}^n$, where $c_{ij}$ denotes the cost of transporting a unit mass from point $x_i$ to $x_j$, we seek a transport plan $M=[m_{ij}]_{i,j=1}^n$, where $m_{ij}$ denotes the amount of mass being transported from $x_i$ to $x_j$, that minimizes the total transportation cost $\trace{C^T M}$ between the two distributions, i.e., the transport plan is required to satisfy $M\ett=\mu_0$ and $M^T\ett=\mu_1$, where $\ett$ denotes an $n\times 1$-vector of ones.

Solving this linear program is computationally expensive for large $n$. It was therefore proposed to regularize the problem by introducing a Kullback-Leibler divergence term (sometimes called entropy term) to the objective \cite{cuturi2013sinkhorn}. 
\begin{definition}
Let $p$ and $q$ be two nonnegative vectors or matrices of the same dimension. The Kullback-Leibler (KL) divergence between $p$ from $q$ is defined as 
\begin{equation*}
H(p|q) :=
 \sum_{i} p_i \log\left( \frac{p_i}{q_i} \right)
\end{equation*}
where $0\log 0$ is defined to be $0$. Note that $H(p|q)$ is jointly convex over $p, q$. See, e.g., \cite{cover2012elements} for more properties and interpretation of the KL divergence.
\end{definition}
The discretized and regularized OMT problem then reads 
\begin{equation} \label{eq:omt_reg}
\begin{aligned} 
  \minwrt[M\in\mathbb{R}^{n\times n}] \quad & \textrm{trace} \left(C^T M\right) + \epsilon H(M|\ett_{n\times n}) \\
\text{subject to} \quad &  M \mathbf{1} = \mu_0, \quad  M^T \mathbf{1} = \mu_1,
\end{aligned}
\end{equation}
where $\epsilon>0$ is a regularization parameter and $\ett_{n\times n}$ denotes an $n\times n$-matrix of ones. The solution to this problem may be found by Sinkhorn iterations, which correspond to the fixed point iteration for the Schrödinger system in \cite{chen2016hilbert}.
\section{Particle dynamics over a Markov chain}
Consider a cloud of $N$ particles, where each particle is evolving according to a Markov chain as described in Section~\ref{ssec:HMM}. Let the vectors $\mu_t \in \RN^{n}$ describe the particle distributions at time $t\in \{0,1\}$, where the $i$-th element $(\mu_t)_i$ denotes the number of particles in state $X_i$ at time $t$. In analogy to the OMT framework, we define the mass transfer matrix $M=[m_{ij}]_{i,j=1}^n$, where $m_{ij}$ denotes the number of particles that transit from state $X_i$ to state $X_j$. Note that the mass transport matrix satisfies $M \ett=\mu_0$ and $M^T \ett=\mu_1$. 

The state transition matrix $A=[a_{ij}]_{i,j=1}^n$ contains the particle transition probabilities. Thus, given the initial state $\mu_0$, the probability of a mass transfer matrix $M$ is
\begin{equation} \label{eq:probabilityM}
P_{\mu_0,A}( M ) = \prod_{i=1}^{n} \left( \binom{(\mu_0)_i}{m_{i1},m_{i2},\dots,m_{in}} \prod_{j=1}^{n} a_{ij}^{m_{ij}} \right),
\end{equation}
where $\binom{\cdot}{\cdot, \ldots, \cdot}$ denotes a multinomial coefficient. The expected distribution at time $t=1$ is then given by $E(\mu_1|\mu_0) = A^T \mu_{0}$. If $\mu_1$ is observed to be different from $A^T \mu_{0}$ a discrete version of the Schrödinger bridge problem can be solved (see Section~\ref{ssec:bridge}). That is, to find the matrix $M$ that maximizes $P_{\mu_0,A}( M )$ subject to that the constraint $M^T\ett=\mu_1$ on the final marginal is satisfied.

If the number of particles is large, then the log-likelihood of \eqref{eq:probabilityM} can be approximated in terms of a KL divergence.
\begin{proposition} \label{thm:loglikelihood}
Given $A$, let $\mu^{(N)}_0\in \RN^{n}$ be a sequence of distributions with $N$ particles, and $M^{(N)}\in \RN^{n\times n}$ be a sequence of mass transfer matrices such that $M^{(N)}\mathbf{1} = \mu^{(N)}_0$ and $\supp(M^{(N)}) \subseteq \supp(\diag(\mu_0^{(N)}) A)$.
Then there exists a constant $C>0$ such that for all $N$ it holds that
\begin{equation*}
\begin{split}
&  \left| \log \left( P_{\mu_0^{(N)}\!,\, A}(M^{(N)}) \right) + H \left( M^{(N)} \, \big| \, \diag(\mu_0^{(N)})A  \right) \right| \\
& \phantom{\leq} \leq C\log(N).
\end{split}
\end{equation*}
\end{proposition}
\vspace{4pt}
\begin{proof}
See appendix.
\end{proof}

Proposition \ref{thm:loglikelihood} implies that for sequences $\mu^{(N)}_0$ and $M^{(N)}$ satisfying the assumptions, if
\begin{equation*}
\frac{1}{N}\mu^{(N)}_0\to \bar \mu_0 \; \mbox{ and } \; \frac{1}{N}M^{(N)}\to \bar M  
\end{equation*}
as  $N\to \infty$, then
\begin{equation*}
 \frac{1}{N}\log \left( P_{\mu_0^{(N)}\!,\, A}(M^{(N)}) \right) \to - H \left( \bar M \, \big| \, \diag(\bar \mu_0)A  \right)
\end{equation*}
as  $N\to \infty$. 
This means that the KL divergence approximates the log-likelihood of $P_{\mu_0, A}(M)$ with increasing accuracy as the number of particles increases. We write this informally as
\begin{equation*}
	P_{\mu_0,A}(M) \sim e^{ - H\left(M | \diag(\mu_0)A\right)}.
\end{equation*}
In terms of large deviation theory, we thus interpret $H(\cdot|\diag(\mu_0)A)$ as the rate function for $P_{\mu_0,A}(\cdot)$. In fact, Proposition~\ref{thm:loglikelihood} can also be derived from a large deviation principle (see, e.g., \cite[Ch.~2.1.1]{dembo2009large}). 

For systems with many particles, we may therefore formulate the problem of finding the most likely mass transfer matrix $M$ between distributions $\mu_0$ and $\mu_1$ with underlying state transition matrix $A$ as the convex optimization problem
\begin{equation} \label{eq:KL_optimization}
\begin{aligned}  \minwrt[M\in\mathbb{R}^{n\times n}] \quad & H \left(M\,|\,\diag(\mu_0)A\right) \\
\text{subject to} \quad &  M \mathbf{1} = \mu_0,\quad M^T \mathbf{1} = \mu_1.
\end{aligned}
\end{equation}
\begin{remark}\label{rem:KL_OMT}
Let $A$ and $\mu_0$ be strictly positive. With the cost matrix
\begin{equation*}
C= -\epsilon \log(\diag(\mu_0)A),
\end{equation*}
the entropy regularized OMT problem \eqref{eq:omt_reg} is equivalent to problem \eqref{eq:KL_optimization}. Note that entropy regularized OMT problems have previously been solved by formulating them in terms of KL-projection problems \cite{benamou2015bregman}.
\end{remark}
\subsection{Connection to Schrödinger bridges}
We note that given the prior distribution $\mu_0$, the objective in \eqref{eq:KL_optimization} may be written as
\begin{equation*}
H(M \mid A) - H(\mu_0 \mid \ett) 
\end{equation*}
where the second term is constant. Hence, if we associate $A$ and $M$ with the measures $d\ccW$ and $d\ccP$ in \eqref{eq:schrodingerbridge}, the problem in Proposition \ref{thm:loglikelihood} relates to a discretized Schrödinger bridge. 
Our problem formulation indeed corresponds to the time and space discrete Schrödinger bridge from \cite{pavon2010discrete}.
To see this, consider a Markov chain of length $T$. Using Proposition~\ref{thm:loglikelihood}, knowing the marginals $\mu_0$ and $\mu_T$, we can find the most likely evolution of the particles between them as the solution to
\begin{equation} \label{eq:opt_markov_chain}
\begin{aligned}
	\minwrt[M_{[1:T]}, \mu_{[1:T-1]}] \ & \sum_{t=1}^{T} H( M_t \,|\,\diag(\mu_{t-1})A) \\
	\text{subject to } \  &  M_t \mathbf{1} = \mu_{t-1}  , \ \  M_t^T \mathbf{1} = \mu_{t}  ,\\
	& \text{for } \  t=1,\dots,T.
\end{aligned}
\end{equation}
Note that for a nonnegative matrix $M_t$ and strictly positive marginal $\mu_{t-1}$, the first constraint asserts that there is a row-stochastic matrix $\bar M_t$ such that $M_t = \diag(\mu_{t-1}) \bar M_t$.
Plugging this expression for the matrices $M_t$ into \eqref{eq:opt_markov_chain} gives
\begin{equation} \label{eq:discrete_schrodinger_bridge}
\begin{aligned}
	\minwrt[\bar M_{[1:T]}, \mu_{[1:T-1]}] \  & \sum_{t=1}^{T} \sum_{i} (\mu_{t-1})_i H \left(  (\bar M_t)_{i \cdot} ,A_{i \cdot}\right)  \\
	\text{subject to } \ & \bar M_t \mathbf{1} = \ett ,\ \  \mu_{t} = \bar M_t^T \mu_{t-1},\\
	& \text{for } \  t=1,\dots,T.
\end{aligned}
\end{equation}
Here $A_{i \cdot}$ denotes the $i$-th row of $A$. This is precisely the formulation of a Schrödinger bridge over a Markov chain from \cite[eq.~(24)]{pavon2010discrete} with time invariant transition probabilities.
In \cite{pavon2010discrete} it is shown that a unique solution to a corresponding Schrödinger system exists if $\mu_T$ is a strictly positive distribution and all elements are strictly positive 
in the matrix $A$ raised to the power $T$.
The solution to the Schrödinger system may be obtained by a fixed point iteration \cite{Georgiou2015discreteSB}, which is linked to the Sinkhorn iterations for entropy regularized OMT problems, cf. Section~\ref{subsec:OMT_intro}. 
 
We note that the optimization problem \eqref{eq:discrete_schrodinger_bridge} is non-convex and will thus work with the formulation \eqref{eq:opt_markov_chain} in the remaining part of this article.
\section{Particle dynamics over hidden Markov chain}
In this section, we extend our framework to the setting of a hidden Markov chain. The initial marginal distribution is assumed to be known.
In case
the hidden states are linked to the observations by a deterministic linear mapping, they may be estimated in a similar fashion as in \cite{elvander2018tracking}.
Here instead, we consider the non-deterministic case where the available observations emerge from the hidden distributions through an observation probability matrix $B\in\RR^{n\times m}$.

Equivalently to the mass transfer plans $M$, define the observation matrix $D\in\RN^{n\times m}$, with entries $d_{jk}$ denoting the number of particles that are in hidden state $X_j$ and observed in state $Y_k$.
Given a hidden state $\mu$, the probability for any observation matrix $D$ is given by $P_{\mu,B}(D)$ as defined in \eqref{eq:probabilityM}.
Hence, the large deviation result in Proposition \ref{thm:loglikelihood} holds for $D$ with rate function $H( \cdot | \diag(\mu) B)$.

Given an initial distribution $\mu_0 \in \RN^{n}$ and a set of measurements $\Phi_1,\dots,\Phi_T \in \RN^m$, we seek the most likely set of matrices $M_1,\dots,M_T$ and $D_1,\dots,D_T$ such that for some set of hidden distributions $\mu_1,\dots,\mu_T$ it holds that
\begin{equation} \label{eq:constraints_MD}
\begin{aligned}
& M_t\ett=\mu_{t-1},\ \ M_t^T\ett=\mu_t,\\
& D_t\ett = \mu_t,\qquad D_t^T \ett =\Phi_t,\quad \text{for }  t=1,\dots,T.
\end{aligned}
\end{equation}
This model is illustrated in Figure~\ref{fig:HMM_model}.
The maximum likelihood solution is obtained by solving the optimization problem
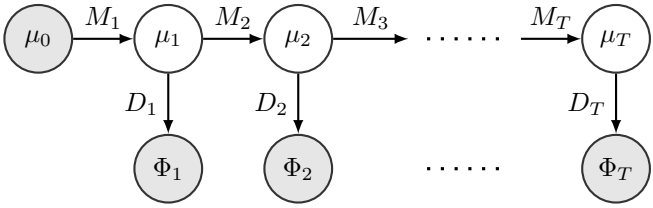
\begin{figure}[tb]
 \centering
 \begin{tikzpicture}
  \tikzstyle{main}=[circle, minimum size = 9mm, thick, draw =black!80, node distance = 8mm]
  \node[main,fill=black!10] (mu0) {$\mu_0$};
  \node[main] (mu1) [right=of mu0]{$\mu_1$};
  \node[main] (mu2) [right=of mu1] {$\mu_2$};
  \node[] (mu3) [right=of mu2] {};
  \node[] (muTm1) [right=of mu3] {};  
  \node[main] (muT) [right=of muTm1] {$\mu_T$};
  
  \node[main,fill=black!10] (phi1) [below=of mu1] {$\Phi_1$};
  \node[main,fill=black!10] (phi2) [right=of phi1,below=of mu2] {$\Phi_2$};
  \node[] (phi3) [right=of phi2] {};
  \node[] (phiTm1) [right=of phi3] {};
  \node[main,fill=black!10] (phiT) [right=of phi3,below=of muT] {$\Phi_T$};

  \draw[->, -latex, thick] (mu0) -- node[above] {$M_1$} (mu1); 
  \draw[->, -latex, thick] (mu1) -- node[above] {$M_2$} (mu2);
  \draw[->, -latex, thick] (mu2) -- node[above] {$M_3$} (mu3); 
  \draw[loosely dotted, very thick] (mu3) -- (muTm1); 
  \draw[->, -latex, thick] (muTm1) -- node[above] {$M_T$} (muT); 
  
  \draw[->, -latex, thick] (mu1) -- node[left] {$D_1$} (phi1);
  \draw[->, -latex, thick] (mu2) -- node[left] {$D_2$} (phi2);
  \draw[->, -latex, thick] (muT) -- node[left] {$D_T$} (phiT);
  \draw[loosely dotted, very thick] (phi3) -- (phiTm1); 
 \end{tikzpicture}
 \caption{Illustration of the hidden Markov model corresponding to \eqref{eq:constraints_MD}.}
\label{fig:HMM_model}
\vspace{-10pt}
\end{figure}
\begin{equation*}
\maxwrt[M_{[1:T]}, D_{[1:T]}, \mu_{[1:T]}] \  \prod_{t=1}^T P_{\mu_{t-1},A} ( M_t ) P_{\mu_{t},B} ( D_t ) 
\end{equation*}
subject to \eqref{eq:constraints_MD}.
From Proposition \ref{thm:loglikelihood} it follows that
\begin{equation*}
\log \left( \prod_{t=1}^TP_{\mu_{t-1},A} ( M_t ) P_{\mu_{t},B} ( D_t ) \right)
\end{equation*}
can be approximated by 
\begin{equation}
\label{eq:HMM_objective} 
 - \sum_{t=1}^T 
 \Big(  
 H(  M_t \mid \diag(\mu_{t-1})A)  +  H(  D_t \mid \diag(\mu_{t})B) 
 \Big)
\end{equation}
when the number of particles is large.
We thus estimate the matrices $M_1,\dots,M_T$, $D_1,\dots,D_T$ and the hidden states $\mu_1,\dots,\mu_T$ by maximizing \eqref{eq:HMM_objective} subject to the constraints \eqref{eq:constraints_MD},
i.e., by solving
\begin{equation}\label{eq:KL_measurement}
\begin{aligned}
\minwrt[M_{[1:T]}, D_{[1:T]}, \mu_{[1:T]}] \ & \sum_{t=1}^{T} H(  M_t \ | \ \diag(\mu_{t-1})A) \\
& + \sum_{t=1}^T H( D_{t} \ | \ \diag(\mu_t) B ) \\
\text{subject to} \quad & M_t \ett = \mu_{t-1} ,\quad M_t^T \ett = \mu_{t}  \\
& D_{t} \ett = \mu_t , \qquad D_{t}^T \ett = \Phi_{t}  \\
& \text{for } t=1,\dots, T.
\end{aligned}
\end{equation}
\begin{remark}
The modeling assumptions leading to this optimization problem require knowledge of the initial distribution as well as the transition and observation probabilities for the hidden Markov model. However, in practice these are typically not known exactly.
In the examples in Section~\ref{sec:simulation}, we illustrate that the estimation is accurate even when there are significant model errors. The generalization of the proposed method to the case where the initial distribution is not known will be discussed in a forthcoming paper.
\end{remark}

\subsection{Computational method}
In this section we develop a numerical method to solve problem
\eqref{eq:KL_measurement}.
To this end, recall from Remark~\ref{rem:KL_OMT} the connection between KL-minimization problems and entropy regularized OMT problems. The latter can be efficiently solved by Sinkhorn iterations\cite{cuturi2013sinkhorn}, which in turn are equivalent to a block coordinate ascent in a dual problem\cite{ringh2017sinkhorn}. Motivated by this, we choose to follow a similar approach. 
\begin{proposition} \label{thm:method}
Let $u_1\in \mR^n$ and $v_t\in \mR^m$, for $t=1,\ldots, T$, be positive initial values, and
iterate the following steps:
\begin{enumerate}
\item[(1)]
$u_1 = e \ett ./ (A w_1)$
\item[(2)]
$v_{t} = e \Phi_{t} ./ \left( B^T (y_t \odot (A w_{t+1})) \right)$
for  $t=1,\dots,T$,
\end{enumerate}
where in each step and for each $t$ in step (2), the vectors $y_t$ and $w_t$ are recursively defined as
\begin{equation*}
\begin{aligned}
& y_1 = A^T (\mu_0 \odot u_1), \\
& y_t = A^T \left(y_{t-1} \odot (Bv_{t-1}) \right), \quad t=2,\ldots,T
\end{aligned}
\end{equation*}
and
\begin{equation*}
\begin{aligned}
& w_T = B v_{T} \\
& w_t = \left( B v_{t} \right) \odot (Aw_{t+1}), \quad t=1, \ldots, T-1.
\end{aligned}
\end{equation*}
In the limit point of the iteration, the estimates for the hidden marginals are then recursively constructed, starting from the known $\mu_0$, as
\begin{equation*}
\mu_t = \diag(w_t)A^T \left(\mu_{t-1} ./ (A w_t) \right), \quad t=1,\dots,T.
\end{equation*}
Furthermore, the corresponding mass transfer matrices are given by
\begin{align*}
M_t &= \frac{1}{e} \diag( \mu_{t-1} \odot u_t ) \, A \, \diag( w_t),\\
D_{t} &= \frac{1}{e} \diag( \mu_t \odot x_{t} ) \, B \, \diag( v_{t}),
\end{align*}
where
\begin{equation*}
\begin{aligned}
& x_{t} = e \ett ./ \left( B v_{t} \right)\\
& u_t = e \ett ./ \left( A w_{t} \right)
\end{aligned}
\end{equation*}
for $t=1,\dots,T.$

\end{proposition}

\begin{proof}
See appendix.
\end{proof}

It is worth noting that intermediate results of the vectors $y_t$ and $w_t$ may be stored, such that the update of $u_1$ requires only one matrix-vector multiplications with $A$, and the update of $v_t$, for any $t=1,\dots,T$, involves two multiplications with $B$ and one with $A$. One iteration sweep, i.e. one update of $u_1$ and the set $v_t$, for $t=1,\dots,T$, thus requires $\mathcal{O}(Tn\max(n,m))$ operations.

\section{Simulations}\label{sec:simulation}
\subsection{Particle dynamics}
Consider a cloud of 1000 particles evolving from an initial distribution $\mu_0\in\RR^n$ with $n=100$ states. The particles transition matrix is given by $\tilde A\in\RR^{n\times n}$ with elements
\begin{equation*}
\tilde a_{ij} \sim \exp\left( \frac{1}{2\sigma_{\tilde a}^2}(i-j-1)^2 \right), \quad \mbox{ with }\sigma_{\tilde a}=0.5,
\end{equation*}
which corresponds to a discretization of a normal distribution $\ccN(1,0.5)$, and thus induces a drift on the dynamics of the cloud. The true dynamics of the particles are assumed to be unknown and instead modeled by a transition matrix $A\in\RR^{n\times n}$ with elements
\begin{equation*}
a_{ij} \sim \exp\left( \frac{1}{2\sigma_a^2}(i-j)^2 \right), \quad \mbox{ with } \sigma_a=2.
\end{equation*}
At each time instance, the particles are observed in $m=5$ bins, where the observation probability matrix $B\in\RR^{n\times m}$ has elements
\begin{equation*}
b_{ij} \sim \exp\left( \frac{1}{2 \sigma_b^2}\left( j-\frac{i+10}{20}\right)^2\right), \quad \mbox{ with }\sigma_b=0.5
.
\end{equation*}
We estimate the flow of the particles and hidden particle distributions solving problem \eqref{eq:KL_measurement} for $T=50$ time instances with the method proposed in Proposition \ref{thm:method}. One estimate is formed using the true initial distribution $\mu_0$ as a prior distribution, and for a second estimate we use a uniform prior.

Figure \ref{fig:simulation_particles} shows the true hidden particle cloud, the corresponding observations, and the two estimates. With full information of the initial states available, the proposed method provides a good estimate of the hidden states despite discrepancies between the true and assumed transition matrices $\tilde A$ and $A$. In the case of no prior information, i.e., the prior distribution is set to be uniform, we see that the estimate converges to the estimate with true prior within a few time steps. This indicates that the proposed method is robust to modeling uncertainties and lack of information in the initial state. 
\subsection{Tracking ensembles over a network}
In this example we consider the problem of tracking a number of indistinguishable agents over a network given measurements from sensors that are distributed around the network. This is inspired by \cite{danielsson2018multi}, where an HMM is used to estimate the flow of a crowd
in an urban environment based on observations generated when cell phones connect to Wi-Fi sensors.
The environment is modeled as a network of nodes and arcs, where the arcs represent walking paths in the area and the nodes are the intersections between the paths.

For this application, the optimization problem \eqref{eq:KL_measurement} needs to be extended to allow for multiple measurements. To this end,
let $\Phi_{st}$ be a set of observations from measurement unit $s$ at time point $t$, for $t=1, \dots, T$, and for $s = 1, \dots ,S$.
We obtain the maximum likelihood solution as the optimal solution to\footnote{This is a convex optimization problem which in principle can be solved with off-the-shelf solvers. In this example we use an efficient algorithm in the spirit of Proposition~\ref{thm:method}, but due to lack of space we defer the exact algorithm to a forthcoming paper.}
\begin{align}
\minwrt[M_{[1:T]}, D_{[1:T],[1:S]}, \mu_{[1:T]}] \ & \sum_{t=1}^{T} H(  M_t \ | \ \diag(\mu_{t-1})A) \nonumber \\
& + \sum_{t=1}^T \sum_{s=1}^S H( D_{st} \ | \ \diag(\mu_t) B_s ) \nonumber\\
\text{subject to} \quad & M_t \ett = \mu_{t-1} ,\quad M_t^T \ett = \mu_{t} \label{eq:KL_multi_measurement} \\
& D_{st} \ett = \mu_t , \qquad D_{st}^T \ett = \Phi_{st} \nonumber \\
& \text{for } t=1,\dots, T, \text{ and } s=1,\dots,S. \nonumber
\end{align}

\begin{figure}[tb]
 \centering
 \includegraphics[width=\columnwidth]{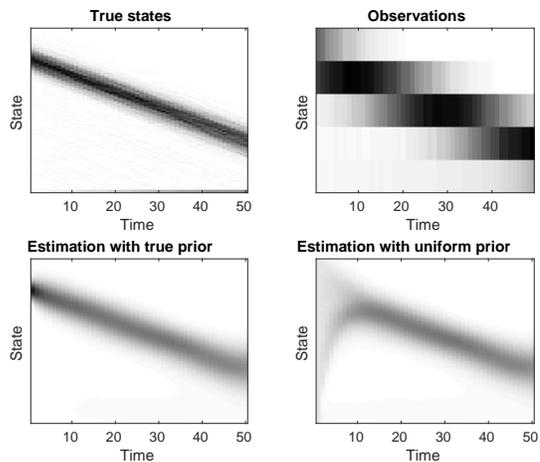}
 \caption{Particle cloud reconstructed from observations.}
\label{fig:simulation_particles}
\vspace{-10pt}
\end{figure}

\begin{figure}[tb]
\centering
\includegraphics[width=\columnwidth]{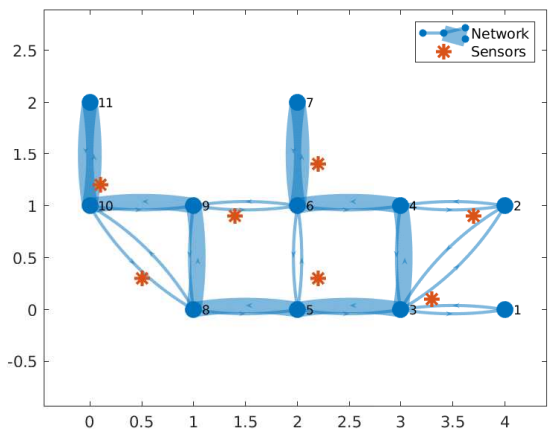}
\caption{Network and sensors.}
\label{fig:network_setting}
\vspace{-10pt}
\end{figure}
Consider a hidden Markov model where the states $X=\{X_1,\ldots, X_n\}$ are the edges in the directed graph $\mathcal{G}=(V,X)$, and where the edge $X_i=(V^{\rm in}_i,V^{\rm out}_i)$ goes from $V^{\rm in}_i\in V$ to $V^{\rm out}_i\in V$. In this example we will use the graph illustrated in Figure \ref{fig:network_setting}, consisting of 11 nodes and $n=28$ edges.
For the true model, transition probabilities are defined according to weights in the graph that represent which walking paths are preferred by the pedestrians. For the model used in the estimation we assume that this information is not known and use uniform weights. 
More specifically, the transition probabilities are given by
\begin{equation*}
\tilde{a}_{ij} =\!
\begin{cases}
0.5, &  \text{ if } j=i \\
\displaystyle{0.5 \,w_{ij}\! \left(\sum_{\{ k: V^{\rm in}_k = V^{\rm out}_i \}} w_{ik}\right)^{-1}} & \text{ if } V^{\rm in}_j = V^{\rm out}_i \\
0, & \text{ else,} 
\end{cases}
\end{equation*}
where $\{w_{ij}\}$ is a set of weights.
For the transition matrix $A$ used in the estimation we assign uniform weights $w_{ij}=1$, for all $(i,j)$ with $V^{\rm in}_j = V_i^{\rm out}$.
For the true transitions the weights are defined as
\begin{equation*}
w_{ij} = \begin{cases}
20,&  X_j \in \mathcal{W} $ and $ V_j^{\rm out}\neq V_i^{\rm in} \\
0, & V_j^{\rm out}= V_i^{\rm in} \\
1, & \text{ else,}
\end{cases}
\end{equation*}
for $(i,j)$ such that $V^{\rm in}_j = V_i^{\rm out}$, and
where $\mathcal{W}$ is the set of edges highlighted in Figure \ref{fig:network_setting}. Note that the second case implies that agents do not transition to the reverse edge in the next time step.

The agents are observed by a set of $S=7$ sensors located at the positions indicated in Figure \ref{fig:network_setting}.
The observation probability for an agent on edge $X_i$ to be detected by a given sensor $s$ is defined as $b^s_{i1} = \min(0.99, 2 e^{-5d} )$, where $d$ denotes the Euclidean distance between the location of $s$ and the midpoint of $X_i$. Consequently the probability of not being detected is $b^s_{i2}=1-b^s_{i1}$.

Given an initial distribution of 100 agents on the edge $(1,3)$, the flow and the measurements for the true ensemble are computed for $T=20$ time steps using the true transition and observation probabilities. 
Then we estimate the flow by solving the optimization problem  \eqref{eq:KL_multi_measurement}. The true and estimated particle distributions are compared for some time steps in Figure \ref{fig:network}, where the width of each edge is proportional to the number of agents on it. As can be seen in the figure the proposed method provides a good estimate also for this example. 
\begin{figure}[tb]
\centering
\includegraphics[width=\columnwidth]{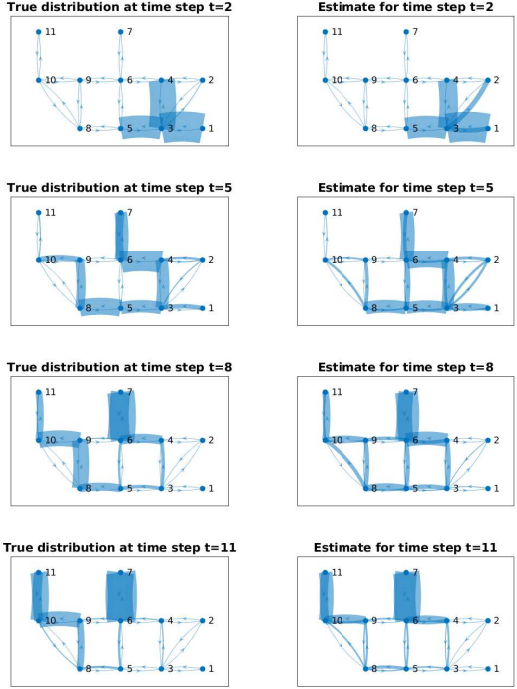}
\vspace{-30pt}
\caption{Ensemble flow over network.}
\label{fig:network}
\vspace{-10pt}
\end{figure}
\section{Conclusions and future directions}
In this work, we propose a method for estimating the flow of an ensemble of particles on a hidden Markov chain. The estimation, which is formulated as a maximum likelihood problem, can be recast as a convex optimization problem, for which we provide an efficient algorithm.

There are several natural directions in which this work can be extended. One restriction in this paper is that the number of agents are fixed and known, thus extensions with a variable number of particles, such as birth/death processes, could be of interest (cf. \cite{dawson1990schrodinger}). Furthermore, the model may be extended to continuous-state dynamics. Another natural direction is to study the connections to reciprocal processes. 
%

%
%
%
%
%
%
\appendix
\subsection{Proof of Proposition \ref{thm:loglikelihood}}
Let $\mu_0^{(N)}$ and $M^{(N)}$ be as described in the statement of the proposition. Moreover, let $\mathcal{Z}^{(N)} = \{ (i,j) \mid m^{(N)}_{ij} \neq 0 \}$, which is non-empty, and let $\mathcal{Z}^{(N)}_i = \{ j \mid (i,j) \in \mathcal{Z}^{(N)} \}$, which is non-empty if and only if $(\mu_0^{(N)})_i > 0$.
Furthermore, let $\mathcal{Y}^{(N)} = \{i \mid (\mu_0^{(N)})_i > 0\}$.
Then, using
Stirling's formula
\begin{equation*}
\sqrt{2\pi} n^{n-1/2} e^{-n} \leq n! \leq e n^{n-1/2} e^{-n},
\end{equation*}
see, e.g., \cite{robbins1955remark},
for $i \in \mathcal{Y}^{(N)}$ the multinomial coefficient in \eqref{eq:probabilityM} can be bounded from above by 
\begin{align*}
& \binom{(\mu_0^{(N)})_i}{m_{i1}^{(N)}, m_{i2}^{(N)}, \dots, m_{in}^{(N)}} \\
& \leq e^{-(|\mathcal{Z}^{(N)}_i| - 1)} \exp\bigg(\sum_{j \in \mathcal{Z}^{(N)}_i} m_{ij}^{(N)}-(\mu_0^{(N)})_i\bigg) \\
& \hspace{20pt} \cdot (\mu_0^{(N)})_i^{(\mu_0^{(N)})_i+\frac{1}{2}} \prod_{j \in \mathcal{Z}^{(N)}_i} (m_{ij}^{(N)})^{-(m_{ij}^{(N)} + \frac{1}{2})} \\
& \leq (\mu_0^{(N)})_i^{(\mu_0^{(N)})_i+\frac{1}{2}} \prod_{j \in \mathcal{Z}^{(N)}_i} (m_{ij}^{(N)})^{-(m_{ij}^{(N)} + \frac{1}{2})},
\end{align*}
where $|\mathcal{Z}^{(N)}_i|$ denotes the cardinality of the set, and where the second inequality follows from the fact that $\sum_{j \in \mathcal{Z}^{(N)}_i} m_{ij} = \sum_{j = 1}^n m_{ij} = (\mu_0)_i$, and that $e^{-(|\mathcal{Z}^{(N)}_i| - 1)} \leq 1$.
Thus, the log-likelihood of the probability for a transfer plan $M^{(N)}$ can be upper-bounded as follows:
\begin{align*}
& \log \left( P_{\mu_0^{(N)},A}(M^{(N)}) \right) \\
& \leq  \sum_{(i,j) \in \mathcal{Z}^{(N)}} \left( m_{ij}^{(N)} \log(a_{ij}) - \left( m_{ij}^{(N)} + \frac{1}{2} \right) \log(m_{ij}^{(N)}) \right)\\
  & \quad + \sum_{i \in \mathcal{Y}^{(N)}} \left( (\mu_0^{(N)})_i + \frac{1}{2} \right) \log( (\mu_0^{(N)})_i)\\
& =   \sum_{(i,j) \in \mathcal{Z}^{(N)}} \left( m_{ij}^{(N)} \log\left( \frac{(\mu_0^{(N)})_i a_{ij}}{m_{ij}^{(N)}} \right) \right)\\
  & \quad - \sum_{(i,j) \in \mathcal{Z}^{(N)}}\frac{1}{2} \log(m_{ij}^{(N)}) +  \sum_{i \in \mathcal{Y}^{(N)}} \frac{1}{2}  \log( (\mu_0^{(N)})_i) \\
& \leq - H\left( M^{(N)},\diag(\mu_0^{(N)})A\right) +  \frac{n}{2} \log(N),
\end{align*}
where the last inequality comes from the second-to-last expression by i) identifying the first term as the KL divergence, ii) noting that the second term is nonpositive, and iii) overestimating the third term by taking $(\mu_0^{(N)})_i = N$ for $i = 1, \ldots, n$.

Similarly, underestimating the multinomial coefficients gives that the log-likelihood can be bounded from below by
\begin{align*}
& \log \left( P_{\mu_0^{(N)},A}(M^{(N)}) \right) \\
& \geq   \sum_{(i,j) \in \mathcal{Z}^{(N)}} \left( m_{ij}^{(N)} \log\left( \frac{(\mu_0^{(N)})_i a_{ij}}{m_{ij}^{(N)}} \right) - \frac{1}{2} \log(m_{ij}^{(N)}) \right)\\
  & \quad +  \sum_{i \in \mathcal{Y}^{(N)}} \frac{1}{2}  \log( (\mu_0^{(N)})_i) - \frac{1}{2} n(n-1) \log(2\pi)\\
  & \geq - H\left( M^{(N)},\diag(\mu_0^{(N)})A\right) \\
 & \phantom{\geq} - \frac{1}{2}\left(n^2 + \frac{n(n-1) \log(2\pi)}{\log(N)}  \right) \log(N).
\end{align*}
By using the two inequalities, the result follows.
\subsection{Proof of Proposition \ref{thm:method}}
We solve \eqref{eq:KL_measurement} by a block coordinate ascent in the dual.
First note that since $\mu_0$, $A$, $B$, and $\Phi_{t}$ are all elementwise nonnegative, so will the optimal solution $M_{[1:T]}^*, D_{[1:T]}^*, \mu_{[1:T]}^*$ also be. We can therefore add the constraint $\mu_{[1:T]} \geq 0$ to \eqref{eq:KL_measurement} without changing the optimal solution.  
For this problem, we relax the constraints in \eqref{eq:KL_measurement} with corresponding dual variables $\lambda_{M_t},\nu_{M_t},\lambda_{D_{t}},\nu_{D_{t}}$. 
Let $\mathfrak{M} := M_{[1:T]}$, $\mathfrak{D} := D_{[1:T]}$,
and define the Lagrangian
\begin{equation*}
\begin{aligned}
& L(\mathfrak{M}, \mathfrak{D}, \mu_{[1:T]}, \lambda_{\mathfrak{M}},\nu_{\mathfrak{M}}, \lambda_{\mathfrak{D}}, \nu_{\mathfrak{D}}) \\
& \phantom{xxx} =  \sum_{t=1}^{T} \Bigg(  \sum_{ij}  m^t_{ij} \log\big(\frac{ m^t_{ij}}{\mu^{t-1}_i a_{ij}} \big) + \lambda_{M_t}^T (\mu_{t-1} - M_t \ett ) \\
& \phantom{xxx =} + \nu_{M_t} (\mu_{t} - M_t^T \ett) +   \sum_{ij}  d^{t}_{ij} \log \big( \frac{ d^{t}_{ij}}{\mu^{t}_i b_{ij}} \big) \\
& \phantom{xxx =} + \lambda_{D_{t}}^T (\mu_t - D_{t} \ett ) + \nu_{D_{t}} (\Phi_{t} - D_{t}^T \ett)  \Bigg). \\
\end{aligned}
\end{equation*}%
Minimizing this with respect to the matrices $M_t$ and $D_{t}$ gives explicit expressions for the optimal solution in terms of $\mu_{[0:T]}$ and the dual variables, i.e.,
\begin{align*}
M_t &= \frac{1}{e} \diag( \mu_{t-1} \odot u_t ) \, A \, \diag( w_t),\\
D_{t} &= \frac{1}{e} \diag( \mu_t \odot x_{t} ) \, B \, \diag( v_{t}),
\end{align*}
where $u_t = \exp(\lambda_{M_t})$, $w_t = \exp(\nu_{M_t})$, $x_{t} = \exp(\lambda_{D_{t}})$ and $v_{t} = \exp(\nu_{D_{t}})$, for $t=1,\dots,T$.
Plugging these into the Lagrangian, we get the modified Lagrangian
\begin{equation*}
\begin{aligned}
& L(\mu_{[1:T]},
u_{\mathfrak{M}}, w_{\mathfrak{M}}, x_{\mathfrak{D}}, v_{\mathfrak{D}})
 =  - \frac{1}{e} \sum_{t=1}^T ( \mu_{t-1} \odot u_t )^T A  w_t \\
& \phantom{x} - \frac{1}{e}  \sum_{t=1}^T  ( \mu_{t} \odot x_{t} )^T B  v_{t} + \log(u_{1})^T \mu_0 \\
& \phantom{x}   + \sum_{t=1}^{T-1} \mu_t^T \left( \log(u_{t+1}) + \log(w_t) + \log(x_{t}) \right) \\
& \phantom{x} + \mu_T^T ( \log(w_T) +  \log(x_{T})  ) + \sum_{t=1}^T \log(v_{t})^T \Phi_{t}.\\
\end{aligned}
\end{equation*}
Noting that since $\mu_{[1:T]}$ occurs linearly in $L$, for the modified dual functional $\inf_{\mu_{[1:T]} \geq 0} L(\mu_{[1:T]}, \lambda_{\mathfrak{M}},\nu_{\mathfrak{M}}, \lambda_{\mathfrak{D}}, \nu_{\mathfrak{D}})$ to be bounded from below, 
the corresponding factors need to be elementwise nonnegative.
In this case the corresponding terms will be zero when taking the infimum, and thus the dual problem is to maximize
\begin{equation}\label{eq:dual_objective}
- \frac{1}{e} (\mu_0 \odot u_1)^T A w_1   + \log(u_1)^T \mu_0 + \sum_{t=1}^T \log(v_{t})^T \Phi_{t} 
\end{equation}
subject to
\begin{equation}\label{eq:constraints1}
\begin{split}
- \frac{1}{e}  \diag(u_{t+1}) A  w_{t+1} - \frac{1}{e}  \diag(x_{t}) B  v_{t} &\\
 + \log(u_{t+1})  + \log(x_{t}) + \log(w_t) & \geq 0
\end{split}
\end{equation}
for $t=1,\dots,T-1$, and
\begin{equation}\label{eq:constraints2}
- \frac{1}{e} \diag(x_{T}) B  v_{T}   + \log(x_{T})+ \log(w_T) \geq 0.
\end{equation}
Note that neither the objective function \eqref{eq:dual_objective}, nor the first $T-1$ constraints \eqref{eq:constraints1} depend on $x_T$. Thus the optimal choice of $x_T$ is the one creating the most slack in the last constraint, i.e., the one maximizing the first two terms of  \eqref{eq:constraints2}.
This is achieved by $x_{T} = e \ett ./(B v_{T})$, and for this choice of $x_T$ the constraint 
\eqref{eq:constraints2} can be replaced by
\begin{equation}\label{eq:constraints3}
- \log(B  v_{T})  + \log(w_T) \geq 0.
\end{equation}
Similarly, the slack in \eqref{eq:constraints1} is maximized by selecting  
\begin{equation*}
\begin{aligned}
& x_{t} = e \ett ./ \left( B v_{t} \right), \quad t=1,\dots,T \\
& u_t = e \ett ./ \left( A w_{t} \right), \quad t=2,\dots,T.
\end{aligned}
\end{equation*}
and thus the constraints \eqref{eq:constraints1} can be replaced by
\begin{equation}\label{eq:constraints4}
- \log( A  w_{t+1}) - \log( B  v_{t})  + \log(w_t) \geq 0, 
\end{equation}
for $t=1,\ldots, T-1$. Next, note that if some of the constraints \eqref{eq:constraints3} and \eqref{eq:constraints4} are not fulfilled with equality, the objective function \eqref{eq:dual_objective} can be improved by increasing the values of the corresponding $v_t$. Therefore, in an optimal point we must have that
\begin{equation*}
\begin{aligned}
w_T & =  B v_{T} \\
w_t & = \left(B v_{t}\right) \odot (A w_{t+1}) \quad \text{for } t=T-1,\dots,1.
\end{aligned}
\end{equation*}
This gives an expression for $w_1$, that depends on $v_t$, $t = 1, \ldots, T$. Inserting this into the objective \eqref{eq:dual_objective} leads to an unconstrained problem that depends only on $u_1$ and $v_{t}$, $t=1,\dots,T$, which we solve using block coordinate ascent.
To this end, the unconstrained objective is first maximized with respect to $u_1$, which gives
\begin{equation*}
u_1 = e \ett ./ (A w_1).
\end{equation*}
Further, note that the gradient of the unconstrained objective with respect to $v_t$ is
\begin{equation*}
- \frac{1}{e} \odot \left( B^T (y_t \odot A w_{t+1}) \right)+ \Phi_{t}./v_{t},
\end{equation*}
where $y_t$ is defined by the recursion
\begin{equation*}
\begin{aligned}
y_1 &= A^T (\mu_0 \odot u_1),\\
y_t &= A^T( y_{t-1} \odot (B v_{t-1}) ), \quad t=2,\dots,T.
\end{aligned}
\end{equation*}
Hence, maximization with respect to $v_t$ is achieved by
\begin{equation*}
v_{t} = e \Phi_{t} ./ \left( B^T (y_t \odot Aw_{t+1} ) \right).
\end{equation*}
As the unconstrained problem is convex and the objective function continuously differentiable, the block coordinate ascent method converges \cite[Prop.~2.7.1]{Bertsekas99}.  
In the limit point, the hidden marginals can be reconstructed from
\begin{equation*}
\begin{aligned}
\mu_t = M_t^T \ett &= \frac{1}{e} \diag(w_t) A^T (\mu_{t-1} \odot u_t)\\
&= \diag(w_t) A^T (\mu_{t-1} ./(Aw_t)).
\end{aligned}
\end{equation*}
\balance
\bibliographystyle{plain}

\bibliography{./ref}

\end{document}